\documentclass[11 pt, reqno]{amsart}

\usepackage{latexsym}
\usepackage{amssymb}
\usepackage{stmaryrd}
\usepackage{mathrsfs}
\usepackage{dsfont}
\usepackage{amsmath}
\usepackage{dsfont}
\usepackage{fancybox,color}
\usepackage{enumerate}
\usepackage[utf8]{inputenc}

\usepackage[colorlinks]{hyperref}
\definecolor{mycolor}{rgb}{0, 0, 1}
\hypersetup{colorlinks,linkcolor={mycolor},citecolor={mycolor},urlcolor={mycolor}}  
\usepackage{eurosym}
\usepackage{url}
\usepackage{graphicx}

\newcommand{\kom}[1]{}
\renewcommand{\kom}[1]{{\bf [#1]}}

 \def\1{\raisebox{2pt}{\rm{$\mathds{1}$}}}

\newtheorem{theorem}{Theorem}[section]

\newtheorem{lemma}[theorem]{Lemma}
\newtheorem{proposition}[theorem]{Proposition}
\theoremstyle{definition}
\newtheorem{definition}[theorem]{Definition}

\newcommand{\R}{{\mathbb R}}

\newcommand{\Om}{\Omega}
\newcommand{\T}{\textup{T}}

\renewcommand{\div}{\operatorname{div}}

\renewcommand{\d}{\hspace{0.05em}\textup{d}}

\newcommand{\abs}[1]{\left|#1\right|}
\newcommand{\norm}[1]{\left|\left|#1\right|\right|}
\newcommand{\ip}[1]{\left\langle #1\right \rangle}

\newcommand{\dist}{\operatorname{dist}}

\newcommand{\esssup}{\operatornamewithlimits{ess\, sup}}
\newcommand{\essinf}{\operatornamewithlimits{ess\,inf}}

\begin{document}

\title[Uniqueness for Trudinger's parabolic equation]{Uniqueness and comparison principle for Trudinger's parabolic equation with Hölder continuous boundary values}
\author[Anttila]{Riku Anttila} \address{Department of Mathematics and Statistics, University of Jyv\"askyl\"a, FI-40014 Jyv\"askyl\"a, Finland} \email{riku.t.anttila@jyu.fi}

\author[Parviainen]{Mikko Parviainen} \address{Department of Mathematics and Statistics, University of Jyv\"askyl\"a, FI-40014 Jyv\"askyl\"a, Finland} \email{mikko.j.parviainen@jyu.fi}

\date{\today}
\keywords{Trudinger's equation, uniqueness, comparison principle, Hölder continuity, doubling of time variables}
\subjclass[2020]{Primary 35K65; Secondary 35K55, 35D30}
\thanks{This work was supported by the Research Council of Finland (project 360185),
and the Finnish Ministry of Education and Culture's Pilot for Doctoral Programmes
(Pilot project Mathematics of Sensing, Imaging and Modelling). We give a special thanks to Peter Lindqvist for sharing his insight about Trudinger's equation, and also for his comments related to this work.
The proof development and the
preparation of the paper were assisted by ChatGPT. The authors take full responsibility for its contents.}

\begin{abstract}
For every $p>1$, we establish uniqueness and comparison principle for Trudinger's parabolic equation
\[
    \partial_t\!\left(\abs{u}^{p-2}u\right)
   = \text{div}\!\left(\abs{\nabla u}^{p-2}\nabla u\right)
    \quad\text{in }\Omega \times(0,\T),
\]
assuming suitable Hölder regularity for the boundary values. Our comparison principle applies to sign-changing weak sub- and supersolutions.
As an application, we obtain uniqueness of weak solutions to Cauchy--Dirichlet problems when the lateral boundary data is suitably Hölder continuous in time. In particular, Lipschitz continuous boundary values are covered.
\end{abstract}
\maketitle

\section{Introduction}

Uniqueness and comparison principle are long standing open problems for sign-changing weak solutions to Trudinger's equation \cite{Trudinger}
\begin{equation}\label{PDE}
    \partial_t\!\left(\abs{u}^{p-2}u\right)
    \,=\,\div(\abs{\nabla u}^{p-2}\nabla u),
    \quad 1<p<\infty.
\end{equation}
Recently, in \cite{ALP26}, uniqueness of weak solutions to Cauchy--Dirichlet problems was established under the natural parabolic Sobolev space and continuity assumptions, with the additional requirements that the boundary value is $C^2$ and $p\geq 2$. However, the methods used there do not seem to extend easily to a much more general situation, and the case $1 < p < 2$ was left open.
In this paper, we establish a novel method and prove \emph{uniqueness} and \emph{comparison principle}, \autoref{thm:unique} and \autoref{thm:comparison-principle}, assuming only certain time H\"older continuity of the boundary values, in addition to the usual assumptions. We cover the full range $p>1$.

\sloppy The difficulty in proving uniqueness and comparison is that, due to the lack of suitable time derivatives of the solutions, we cannot directly use the seemingly natural test functions.
In \cite{ALP26} we bypassed these difficulties using Otto's \cite{Otto} version of Kru\v{z}kov's \cite{kruzkov70} doubling of time variables, which Otto used to prove a comparison principle for time-independent boundary values. For time-dependent boundary values, however, an application of Otto's method in obtaining a comparison principle required a strict gap between the boundary values, see \cite[Theorem 3.1]{ALP26}. To prove uniqueness from this ``gap comparison principle'', we used an approximation argument of \cite{LPSSarsa} which relies on the $C^2$ assumption.

In order to relax the $C^2$ assumption of \cite{ALP26}, the core idea in the present work is to test the equation with functions that involve
\[
    \eta'\!\left(\frac{u(x,t)-v(x,s)}{\delta(x)}\right),
\]
where $\delta(x)$ ranges over carefully chosen functions depending on the spatial variable $x$. In contrast, $\delta(x)$ are constants in \cite{Otto,ALP26}. This allows us to treat the regions close to and far from the boundary separately. Indeed, these functions are chosen so that they are large enough on the boundary to accommodate the difference between the times $t$ and $s$, but much smaller in the interior. Its spatial variation introduces an additional error, which we absorb against the positive monotonicity term before estimating the remainder. This gives a direct comparison principle without a strict gap between the boundary values or auxiliary approximating solutions.

Uniqueness and comparison for non-negative solutions to Trudinger's equation have been studied in \cite{LL22,BS24} with certain strict positivity assumptions. Some methods used in these works originate from \cite{IMJ97}. For $p\geq 2$, \cite{LPSSarsa} proves uniqueness among non-negative weak solutions with $C^2$ boundary data. A Galerkin approximation, cf. \cite{altl83}, is used there to obtain the basic estimates. For sign-changing solutions with time-dependent boundary data, the $C^2$ uniqueness result in \cite{ALP26} provides the starting point for the present work.

Furthermore, Trudinger's equation by now has an extensive literature on regularity theory.
For positive solutions, \cite{LPSSiltakoski} establishes local Lipschitz continuity in the spatial variables, while the Schauder estimates of \cite{BDGLS26} imply that the spatial gradient is locally H\"older continuous. Regularity results also extend beyond the positive case, see \cite{BDL21} for H\"older continuity of sign-changing solutions and \cite{GV06,KK07} for Harnack inequalities of non-negative solutions. The relation between the weak and viscosity formulations has also been studied. These notions of solution coincide in the positive case by \cite{LPSSiltakoski}, and in the non-negative case for $p\geq2$ by \cite{LPS26c}.

\subsection*{Plan of the proof}
In this outline we restrict to $p\geq 2$ for simplicity, although the proofs cover the full range $1<p<\infty$. Let  $\Om_\T=\Om\times(0,\T)$. We use the Heaviside approximation
\[
    \eta'(z) :=
    \begin{cases}
        0,&z\leq1,\\
        z-1,&1<z<2,\\
        1,&z\geq2.
    \end{cases}
\]
As $\delta\downarrow0$, $\eta'(z/\delta)$ converges to the indicator of $\{z>0\}$. Thus $\eta'((u-v)/\delta)$ singles out the positivity set of $u-v$.

We first illustrate comparison in the favorable case when suitable time derivatives exist, following a well known argument. Suppose that $u$ is a subsolution, $v$ is a supersolution, and $u\leq v$ on the parabolic boundary. Assume for this explanation that
\[
    \partial_t(\abs{u}^{p-2}u),\quad
    \partial_t(\abs{v}^{p-2}v)\in L^1(\Om_\T).
\]
For a constant $\delta>0$, we use the same non-negative test function
\[
    \eta'\!\left(\frac{u-v}{\delta}\right)\zeta(t),
    \qquad 0\leq\zeta\in C_0^\infty(0,\T),
\]
in the two inequalities and subtract. This gives
\begin{align*}
    &\int \limits_0^\T\!\!\!\int \limits_\Om
    \eta'\!\left(\frac{u-v}{\delta}\right)
    \partial_t\!\left(\abs{u}^{p-2}u-\abs{v}^{p-2}v\right)\zeta\,\d x\d t\\
    &\quad+\int \limits_0^\T\!\!\!\int \limits_\Om
    \frac{\zeta}{\delta}\eta''\!\left(\frac{u-v}{\delta}\right)
    \ip{\abs{\nabla u}^{p-2}\nabla u-\abs{\nabla v}^{p-2}\nabla v,
    \nabla u-\nabla v}\,\d x\d t\leq0.
\end{align*}
The second integral is non-negative by monotonicity. Letting $\delta\downarrow0$ and integrating by parts in time, we obtain
\[
    \int \limits_0^\T\!\!\!\int \limits_\Om
    \left(\abs{u}^{p-2}u-\abs{v}^{p-2}v\right)^+\zeta'\,\d x\d t\geq0.
\]
Thus the spatial integral of the positive part is non-increasing, and the initial ordering gives $u\leq v$.

However, these time derivatives are not assumed in the present paper. Instead, under the hypotheses of Theorem~\ref{thm:comparison-principle}, we double the time variable and use regularizations of
\[
    \eta'\!\left(\frac{u(x,t)-v(x,s)}{\delta(x)}\right)
    \zeta(t)\phi_\varepsilon(t-s),
\]
where $\phi_\varepsilon$ is the rescaled mollifier defined in Section~\ref{sec:errors}, and $\delta$ is a positive Lipschitz function. The details of this testing argument are given in Lemma~\ref{lemma:key} and Proposition~\ref{prop:doubling}.

There are two errors to control: the temporal error $E_{\mathrm T}$ and the spatial error $E_{\mathrm S}$. The temporal error comes from replacing the time primitives appearing in Proposition~\ref{prop:doubling} by $(\abs{u(t)}^{p-2}u(t)-\abs{v(s)}^{p-2}v(s))^+$; see Lemma~\ref{lemma:error-1}. The proof of Proposition~\ref{prop:bound-I} gives
\[
    E_{\mathrm T}\leq\frac{C}{\varepsilon} \int \limits_\Om\delta\,\d x
\]
for small $\varepsilon$, where the factor $1/\varepsilon$ comes from differentiating the time mollifier. This error already occurs when $\delta$ is constant.

The additional spatial error comes from differentiating $\delta$. By the chain rule we have
\begin{align*}
    &\nabla\eta'\!\left(\frac{u(t)-v(s)}{\delta}\right)\\
    &\quad=\frac1\delta\eta''\!\left(\frac{u(t)-v(s)}{\delta}\right)
    \left(\nabla u(t)-\nabla v(s)
    -\frac{\nabla \delta}{\delta} (u(t) - v(s)) \right).
\end{align*}
Above we omit the spatial variable to keep the equation compact.
The last term need not have a favorable sign, meaning we need to estimate it. Furthermore, it is essential to absorb some of its contribution against the positive monotonicity term before discarding that term, so this point requires quite subtle estimation. Nevertheless, in Lemma ~\ref{lemma:error-2}, we obtain the spatial error term
\[
    E_{\mathrm S}=
    \left(\, \int \limits_\Om\frac{\abs{\nabla\delta}^p}{\delta^{p/2}}\,\d x\right)^{2/p}.
\]
Combining the two estimates, and ensuring that our test function with doubled time variable has zero lateral trace, Proposition~\ref{prop:error-3} gives the inequality
\[
\begin{aligned}
    &\int \limits_0^\T\!\!\!\int \limits_0^\T\!\!\!\int \limits_\Om
    \left(\abs{u(t)}^{p-2}u(t)-\abs{v(s)}^{p-2}v(s)\right)^+
    \zeta'(t)\phi_\varepsilon(t-s)\,\d x\d t\d s\\
    &\qquad\geq -C(E_{\mathrm T}+E_{\mathrm S}).
\end{aligned}
\]
We note that Hölder continuity and careful selections of $\delta$ are both crucial to guarantee that the lateral trace is zero. Moreover, in order to control the error terms, $\delta$ must be much smaller in the interior than near the boundary. The precise choices for $\delta$ are given in Lemma~\ref{lemma:delta-2}, and the resulting error term bounds are given in Proposition~\ref{prop:bound-I}. The error terms are handled in the proof of Theorem \ref{thm:comparison-principle}.

\section{Preliminaries}\label{sec:Preli}
Throughout the paper we have fixed $p > 1$, $\T>0$ and let $\Om \subset \R^n$ be a bounded Lipschitz domain. We use the standard notation
\[
    \Om_\T := \Om \times (0,\T) \quad \text{and} \quad \partial_{\mathcal{P}} \Om_\T := (\partial \Om \times [0,\T]) \cup (\Om \times \{0\}).
\]
The convex primitive $\eta : \R \to \R$ of the Heaviside approximation used above, normalized by $\eta(0)=0$, is given by
\[
    \eta(z) := \begin{cases}
        0 & \text{ when } z \leq 1,\\
        \vspace{5pt}
            \dfrac{(z-1)^2}{2} & \text{ when } 1 < z < 2,\\
            \vspace{5pt}
        z - \dfrac{3}{2} & \text{ when } z \geq 2.
        \end{cases}
\]

\subsection{Terminology}
We recall some standard terminology and notation of parabolic equations.
\begin{definition}
    The \emph{parabolic Sobolev space} $L^p(0,\T;W^{1,p}(\Omega))$ consists of measurable functions $u : \Om_\T \to [-\infty,\infty]$ satisfying the following.
    \begin{enumerate}
        \item For a.e.\ $t \in (0,\T)$ the function
        $x \mapsto u(x,t)$ lies in $W^{1,p}(\Omega)$.
        \item The function $u$ has finite \emph{parabolic Sobolev norm}
        \[
            \left(\int \limits_{0}^\T  \!\!\! \int \limits_{\Om} \abs{u}^p + \abs{\nabla u}^p \, \d x \d t\right)^{\frac{1}{p}} < \infty.
        \]
    \end{enumerate}
    Here $\nabla u$ denotes the spatial weak gradient of $u$. The space $L^p(0,\T;W^{1,p}_0(\Om))$ is defined analogously.
\end{definition}

    \begin{definition}\label{def:compare}
        Given two functions $u,v \in L^p(0,\T;W^{1,p}(\Om))$ the notation $u \leq v$ on $\partial_\mathcal{P} \Om_\T$ means that the following hold.
    \begin{enumerate}
        \item $(u-v)^+ \in L^p(0,\T;W_0^{1,p}(\Om))$.
        \item It holds that
        \[
        \limsup_{h \downarrow 0} \frac{1}{h}\int \limits_{h}^{2h} \!\!\! \int \limits_{\Om} \left(\abs{u}^{p-2}u -\abs{v}^{p-2}v\right)^+ \d x \d t = 0.
        \]
    \end{enumerate}
    If a third function $\psi \in L^p(0,\T;W^{1,p}(\Om))$ is given we also write $u \leq \psi \leq v$ on $\partial_\mathcal{P} \Om_\T$ with the obvious interpretation.
    \end{definition}

    \begin{definition}\label{def:solution}
        Let $u \in L^p(0,\T;W^{1,p}(\Om))$. We say that $u$ is a \emph{weak subsolution} to \eqref{PDE} if it holds for every non-negative $\varphi \in C_0^\infty( \Om_\T )$ that
            \begin{equation*}
                \int \limits_0^\T \!\!\! \int \limits_\Om -\abs{u}^{p-2}u \, \partial_t \varphi + \ip{\abs{\nabla u}^{p-2}\nabla u, \nabla \varphi }\, \d x \d t \leq 0.
            \end{equation*}
        If $-u$ is a weak subsolution \eqref{PDE} then $u$ is called a \emph{weak supersolution}. Finally, $u$ is a \emph{solution} to \eqref{PDE} given that it is both a weak subsolution and a weak supersolution to \eqref{PDE}.
    \end{definition}

    \begin{definition}\label{def:cauchy-dirichlet}
        Let $\psi \in C(\overline{\Om_\T}) \cap L^p(0,\T;W^{1,p}(\Om))$ and $u$ be a solution to \eqref{PDE}. We say that $u$ solves the \emph{Cauchy--Dirichlet} problem with boundary values $\psi$ if the following hold.
        \begin{enumerate}
            \item $u-\psi \in L^p(0,\T;W_0^{1,p}(\Om))$.
            \item It holds that
            \[
                \lim_{t \downarrow 0} \int \limits_{\Om} \abs{ \abs{u(x,t)}^{p-2}u(x,t) - \abs{\psi(x,t)}^{p-2}\psi(x,t)}  \d x = 0. 
            \]
        \end{enumerate}
\end{definition}

And finally, given $\alpha \in (0,1]$ and a boundary value $\psi \in C(\overline{\Om_\T})$, we say that $\psi$ is \emph{$\alpha$-H\"older continuous in time} if there is $H>0$ such that
    \[
        \abs{\psi(x,t)-\psi(x,s)}\leq H\abs{t-s}^\alpha
        \quad\text{for all }x\in\partial\Om\text{ and }t,s\in(0,\T).
    \]

\section{Doubling of time variables}\label{sec:Otto}
We begin the proof of Theorem \ref{thm:comparison-principle} by deriving fundamental integral inequalities through Otto's doubling of time variables method.

Let $\tilde{\eta}(z):=\eta(-z)$. For $\delta>0$, define
\[
    q_\delta(z,w) := (p-1)\int \limits_w^z \eta'\left(\frac{\tau-w}{\delta}\right) \abs{\tau}^{p-2}\, \d \tau.
\]
We define $\tilde{q}_\delta$ by replacing $\eta$ with $\tilde{\eta}$ in this formula. Note that $\tilde{q}_\delta$ can also be given by $\tilde{q}_\delta(w,z) = q_\delta(-w,-z)$.
When $\delta$ depends on $x$, the notation $q_\delta$ and $\tilde{q}_\delta$ is understood pointwise with parameter $\delta(x)$. 
We also sometimes suppress the spatial $x$ and write $u(t)=u(x,t)$ and $v(s)=v(x,s)$.

\begin{lemma}\label{lemma:key}
    Let $u$ be a weak subsolution to \eqref{PDE}, $\zeta \in C_0^\infty(0,\T)$ be non-negative and $\delta : \overline{\Om} \to \R$ be a positive Lipschitz function. Also let $f \in W^{1,p}(\Om)$ such that 
    \[
        x \mapsto \eta'\left( \frac{u(x,t)-f(x)}{\delta(x)} \right)\zeta(t)
    \]
    lies in $W_0^{1,p}(\Om)$ for almost every $t \in (0,\T)$. Then
    \begin{align*}
         & \int \limits_0^\T \!\!\! \int \limits_\Om q_\delta(u(t),f) \partial_t \zeta(t) \, \d x \d t \geq \int \limits_0^\T \!\!\! \int \limits_\Om \frac{\zeta(t)}{\delta} \eta'' \! \left( \frac{u(t) - f}{\delta} \right)\\
         & \hspace{2.5cm} \times \ip{ \abs{\nabla u(t)}^{p-2}\nabla u(t) , \nabla u(t) - \nabla f -\frac{\nabla\delta}{\delta} (u(t) - f) } \d x \d t.
        \end{align*}
\end{lemma}

\begin{proof}
    The proof in the case where $\delta$ is a constant function is treated in \cite[Lemma 3.3]{ALP26}. The general version follows from the same proof with an obvious modification involving  
    $$\nabla  \eta'\left( \frac{u(x,t)-f(x)}{\delta(x)} \right)$$
     and a chain rule.
\end{proof}

\begin{proposition}\label{prop:doubling}
    Let $u$ be a weak subsolution and $v$ be a weak supersolution to \eqref{PDE}, respectively. Also let $\delta : \overline{\Om} \to \R$ be a positive Lipschitz function and $\gamma \in C_0^\infty((0,\T) \times (0,\T))$ be non-negative.
    Suppose that
    \[
        x \mapsto \eta'\left( \frac{u(x,t)-v(x,s)}{\delta(x)} \right)\gamma(t,s)
    \]
    lies in $W_0^{1,p}(\Om)$ for almost every $t,s \in (0,\T)$. Then
    \begin{align*}
        & \quad \, \int \limits_0^\T \!\!\!  \int \limits_0^\T \!\!\! \int \limits_\Om q_\delta(u(t),v(s)) \partial_t \gamma(t,s) + \tilde{q}_\delta(v(s),u(t)) \partial_s \gamma(t,s) \d x \d t \d s \\
        &\geq \quad \int \limits_0^\T \!\!\!  \int \limits_0^\T \!\!\! \int \limits_\Om \frac{\gamma(t,s)}{\delta}\eta''\left( \frac{u(t)-v(s)}{\delta} \right)\\
        & \times \Bigg(\ip{\abs{\nabla u(t)}^{p-2}\nabla u(t) - \abs{\nabla v(s)}^{p-2}\nabla v(s), \nabla u(t) - \nabla v(s)}\\
        & \hspace{.5cm} -\ip{ \abs{\nabla u(t)}^{p-2}\nabla u(t) - \abs{\nabla v(s)}^{p-2}\nabla v(s), \frac{\nabla\delta}{\delta} (u(t) - v(s))}\Bigg) \d x \d t \d s.
    \end{align*}
\end{proposition}

\begin{proof}
    For almost every $s \in (0,\T)$ we can apply Lemma \ref{lemma:key} with subsolution $u$, $\zeta(t) = \gamma(t,s)$ and $f(x) = v(x,s)$. This yields
    \begin{align*}
         & \int \limits_0^\T \!\!\! \int \limits_\Om q_\delta(u(t),v(s)) \partial_t \gamma(t,s) \, \d x \d t \geq \int \limits_0^\T \!\!\! \int \limits_\Om \frac{\gamma(t,s)}{\delta} \eta'' \! \left( \frac{u(t) - v(s)}{\delta} \right)\\
         & \hspace{1.8cm} \times \ip{ \abs{\nabla u(t)}^{p-2}\nabla u(t) , \nabla u(t) - \nabla v(s) -\frac{\nabla\delta}{\delta} (u(t) - v(s)) } \d x \d t.
        \end{align*}
    Then we interchange the role of $t$ and $s$ and apply Lemma \ref{lemma:key} again. In this case the parameters are subsolution $-v$, $\zeta(s) = \gamma(t,s)$, $f(x) = -u(x,t)$. By noting that $\tilde{q}_\delta(v(s),u(t)) = q_\delta(-v(s),-u(t))$ we obtain
    \begin{align*}
         & \int \limits_0^\T \!\!\! \int \limits_\Om \tilde{q}_\delta(v(s),u(t)) \partial_s \gamma(t,s) \, \d x \d s \geq -\int \limits_0^\T \!\!\! \int \limits_\Om \frac{\gamma(t,s)}{\delta} \eta'' \! \left( \frac{u(t) - v(s)}{\delta} \right)\\
         & \hspace{1.8cm} \times \ip{ \abs{\nabla v(s)}^{p-2}\nabla v(s) , \nabla u(t) - \nabla v(s) -\frac{\nabla\delta}{\delta} (u(t) - v(s)) } \d x \d s.
        \end{align*}
    Then we integrate the first integral inequality with respect to $s \in (0,\T)$, the second with respect to $t \in (0,\T)$, and add these together. From this arises  the desired integral inequality.
\end{proof}

\section{Error terms}\label{sec:errors}
The doubling of time variables in Proposition \ref{prop:doubling} creates some technical error terms that require a careful treatment. The purpose of this section is to collect and simplify these terms.

We first introduce auxiliary two-variable functions which are slightly modified from those in Otto's work; see also \cite{ALP26}. Throughout the remaining paper let $\phi$ denote the standard mollifier
\[
    \phi(\tau) :=
    \begin{cases}
        C\exp\left(\dfrac{-1}{1-\abs{\tau}^2}\right) & \text{ if } \abs{\tau} < 1\\
        0 & \text{ if } \abs{\tau} \geq 1.
    \end{cases}
\]
The constant $C > 0$ is chosen so that $\phi$ is a probability density.
For $\varepsilon > 0$ we also define $\phi_\varepsilon(\tau) := (1/\varepsilon)\phi(\tau/\varepsilon)$. And lastly, given $\zeta \in C_0^\infty(0,\T)$, we let
\[
    \zeta_\varepsilon(t,s) := \zeta(t)\phi_\varepsilon(t-s).
\]

We first estimate the left-hand side of Proposition \ref{prop:doubling} and identify the \emph{temporal error term} $E_{\mathrm T}$. The proof is a straightforward application of Lemma \ref{lemma:q-error}.

\begin{lemma}\label{lemma:error-1}
    Let $u,v \in L^p(\Om_\T)$ such that $u$ is essentially bounded from above and $v$ is essentially bounded from below. Also let $\delta : \overline{\Om} \to \R$ be a positive Lipschitz function and $\zeta \in C_0^\infty(0,\T)$.
    Then for all $\varepsilon>0$ it holds that
    \begin{align*}
        & \quad \, \int \limits_0^\T \!\!\!  \int \limits_0^\T \!\!\! \int \limits_\Om q_\delta(u(t),v(s)) \partial_t \zeta_\varepsilon(t,s) + \tilde{q}_\delta(v(s),u(t)) \partial_s \zeta_\varepsilon(t,s) \d x \d t \d s\\
        &\leq \quad \int \limits_0^\T \!\!\!  \int \limits_0^\T \!\!\! \int \limits_\Om \left( \abs{u(t)}^{p-2}u(t) - \abs{v(s)}^{p-2}v(s) \right)^+ \zeta'(t)\phi_\varepsilon(t-s)\,\d x \d t \d s + C \cdot E_{\mathrm T}
    \end{align*}
    where $E_{\mathrm T}$ is the integral
    \[
        E_{\mathrm T} := \int \limits_0^\T \!\!\!  \int \limits_0^\T \!\!\! \int \limits_\Om \delta^{\min\{ p-1,1 \}}\left( \abs{\phi_\varepsilon'(t-s)}\abs{\zeta(t)} + \phi_\varepsilon(t-s)\abs{\zeta'(t)} \right) \,\d x \d t \d s.
    \]
    The constant $C > 0$ depends only on $p$ and $u,v$.
\end{lemma}

\begin{proof}
    By computing the partial derivatives of $\zeta_\varepsilon(t,s) = \zeta(t)\phi_\varepsilon(t-s)$ using the Leibniz rule we can rewrite
    \begin{align*}
        & \quad \, \int \limits_0^\T \!\!\!  \int \limits_0^\T \!\!\! \int \limits_\Om q_\delta(u(t),v(s)) \partial_t \zeta_\varepsilon(t,s) + \tilde{q}_\delta(v(s),u(t)) \partial_s \zeta_\varepsilon(t,s) \d x \d t \d s\\
        &= \quad  \int \limits_0^\T \!\!\!  \int \limits_0^\T \!\!\! \int \limits_\Om q_\delta(u(t),v(s)) \zeta'(t)\phi_\varepsilon(t-s)\\
         & \hspace{2cm} + \left\{q_\delta(u(t),v(s)) - \tilde{q}_\delta(v(s),u(t)) \right\} \zeta(t)\phi_\varepsilon'(t-s) \d x \d t \d s.
    \end{align*}
    We can further rewrite the right-hand side as
    \begin{align*}
        & \quad \, \int \limits_0^\T \!\!\!  \int \limits_0^\T \!\!\! \int \limits_\Om \left( \abs{u(t)}^{p-2}u(t) - \abs{v(s)}^{p-2}v(s) \right)^+ \zeta'(t)\phi_\varepsilon(t-s)\,\\
        & \hspace{.5cm} +
        \left\{q_\delta(u(t),v(s)) - \left( \abs{u(t)}^{p-2}u(t) - \abs{v(s)}^{p-2}v(s) \right)^+\right\} \zeta'(t)\phi_\varepsilon(t-s)\\
        & \hspace{2cm} + \bigg\{q_\delta(u(t),v(s)) - \tilde{q}_\delta(v(s),u(t)) \bigg\} \zeta(t)\phi_\varepsilon'(t-s) \d x \d t \d s.
    \end{align*}
    Then we apply Lemma \ref{lemma:q-error} with $M := \norm{u^+}_\infty+\norm{(-v)^+}_\infty$ and get
    \[
        \abs{q_\delta(u(t),v(s)) - \left(\abs{u(t)}^{p-2}u(t)-\abs{v(s)}^{p-2}v(s)\right)^+ } \leq K \delta^{\min \{p-1,1\} }
    \]
    and also
    \[
        \abs{q_\delta(u(t),v(s))-\tilde q_\delta(v(s),u(t))}
        \leq 2K\delta^{\min \{p-1,1\}}.
    \]
    By applying these two estimates to the integrals displayed above we obtain the desired integral inequality with $C:=2K$.
\end{proof}

Then we estimate the right-hand side of Proposition \ref{prop:doubling} and isolate the \emph{spatial error term} $E_{\mathrm S}$.
In the proof below, one might be tempted to use the standard monotonicity inequality
\[
\bigg\langle|\nabla u(t)|^{p-2}\nabla u(t)-|\nabla v(s)|^{p-2}\nabla v(s),\nabla u(t)- \nabla v(s) \bigg \rangle \geq0    
\]
immediately. It is, however, crucial to apply the sharper estimate provided by Lemma \ref{lemma:L-p}.
Similar improved estimates play a role in proving comparison principle in \cite[Lemma 2.1]{Att12} and \cite[Lemma 3.2]{Sil21}.
Using monotonicity alone the spatial error term we can obtain is 
\[
    \left(\int \limits_\Om \frac{\abs{\nabla \delta}^p}{\delta^p}\, \d x\right)^{1/p} = \left(\int \limits_\Om \abs{\nabla \log \delta}^p \d x\right)^{1/p}.
\]
Proving comparison principle with this error term seems hopeless.

\begin{lemma}\label{lemma:error-2}
    Let $u,v \in L^p(0,\T;W^{1,p}(\Om))$. Also let $\delta : \overline{\Om} \to \R$ be a positive Lipschitz function and $\zeta \in C_0^\infty(0,\T)$ be non-negative. Then it holds for all $\varepsilon>0$ that
    \begin{align*}
        -C \cdot E_{\mathrm S} &\leq \int \limits_0^\T \!\!\!  \int \limits_0^\T \!\!\! \int \limits_\Om  \frac{\zeta_\varepsilon(t,s)}{\delta}\eta''\left( \frac{u(t)-v(s)}{\delta} \right)\\
        & \hspace{-1cm} \times \Bigg( \ip{ \abs{\nabla u(t)}^{p-2}\nabla u(t) - \abs{\nabla v(s)}^{p-2}\nabla v(s), \nabla u(t) - \nabla v(s)}\\
        & \hspace{-.75cm} - \ip{ \abs{\nabla u(t)}^{p-2}\nabla u(t) - \abs{\nabla v(s)}^{p-2}\nabla v(s), \frac{\nabla\delta}{\delta} (u(t) - v(s))} \Bigg)  \d x \d t \d s,
    \end{align*}
    where $E_{\mathrm S}$ is given for $p \geq 2$ by
    \[
        E_{\mathrm S} := \left(\int_\Om \frac{\abs{\nabla \delta}^p}{\delta^{p/2}} \, \d x \right)^{2/p}
    \]
    and for $1 < p < 2$ by
    \[
        E_{\mathrm S} := \int_\Om \frac{\abs{\nabla \delta}^p}{\delta} \, \d x.
    \]
    Here $C > 0$ depends only on $p$, $\T$, $u,v$ and $\zeta$.
\end{lemma}

\begin{proof}
    We consider $p \geq 2$ first. By applying Lemma \ref{lemma:L-p} with $a = \nabla u(t)$, $b=\nabla v(s)$ and $c=(\nabla \delta/\delta)(u(t) - v(s))$ it follows that
    \begin{align*}
        & \quad \Bigg( \ip{ \abs{\nabla u(t)}^{p-2}\nabla u(t) - \abs{\nabla v(s)}^{p-2}\nabla v(s), \nabla u(t) - \nabla v(s)}\\
        & \hspace{.5cm} - \ip{ \abs{\nabla u(t)}^{p-2}\nabla u(t) - \abs{\nabla v(s)}^{p-2}\nabla v(s), \frac{\nabla\delta}{\delta} (u(t) - v(s))} \Bigg)\\
        &\geq \quad -K\frac{\abs{\nabla \delta}^2}{\delta^2}\abs{u(t) - v(s)}^2 \left( \abs{\nabla u(t)}^p + \abs{\nabla v(s)}^p  \right)^{(p-2)/p}.
    \end{align*}
    Hence for $p > 2$ we estimate
    \begin{align*}
        & \quad \, \int \limits_0^\T \!\!\!  \int \limits_0^\T \!\!\! \int \limits_\Om \frac{\zeta_\varepsilon(t,s)}{\delta}\eta''\left( \frac{u(t)-v(s)}{\delta} \right)\\
        & \hspace{2cm} \times \frac{\abs{\nabla \delta}^2}{\delta^2}\abs{u(t) - v(s)}^2 \left( \abs{\nabla u(t)}^p + \abs{\nabla v(s)}^p  \right)^{(p-2)/p} \,\d x \d t \d s \\
        & \leq \quad 4\int \limits_0^\T \!\!\!  \int \limits_0^\T \!\!\! \int \limits_\Om \zeta_\varepsilon(t,s)\frac{\abs{\nabla \delta}^2}{\delta} \left( \abs{\nabla u(t)}^p + \abs{\nabla v(s)}^p  \right)^{(p-2)/p} \,\d x \d t \d s\\
        & \leq \quad 4\left(\int \limits_0^\T \!\!\! \int \limits_0^\T \zeta_\varepsilon(t,s) \,\d t \d s \int \limits_\Om \frac{\abs{\nabla \delta}^p}{\delta^{p/2}} \d x \right)^{2/p} \\
        & \hspace{3cm} \times \left( \int \limits_0^\T \!\!\!  \int \limits_0^\T \!\!\! \int \limits_\Om \left( \abs{\nabla u(t)}^p + \abs{\nabla v(s)}^p\right) \zeta_{\varepsilon}(t,s) \,\d x \d t \d s \right)^{(p-2)/p}\\
        & \leq \quad 4\T^{2/p} \norm{\zeta}_{\infty} \left(\norm{\nabla u}_{L^p(\Om_\T)} + \norm{\nabla v}_{L^p(\Om_\T)}\right)^{p-2} \left(\int \limits_\Om \frac{\abs{\nabla \delta}^p}{\delta^{p/2}} \d x \right)^{2/p}.
    \end{align*}
    In the first inequality we used the fact that the integrand vanishes unless $\abs{u(t) - v(s)} \leq 2\delta$.
    In the second inequality we used Hölder's inequality with exponent $p/2$ over the measure whose density is $\zeta_\varepsilon \,\d x \d t \d s$. The last inequality is deduced from $\int \phi_\varepsilon \d \tau = 1$.
    A combination of the previous two displays proves the case $p > 2$. The case $p = 2$ is easy to deduce from the above computations.

    Then we consider $1 < p < 2$. By applying Lemma \ref{lemma:L-p} with the same parameters as above, we get
    \begin{align*}
        & \quad \Bigg( \ip{ \abs{\nabla u(t)}^{p-2}\nabla u(t) - \abs{\nabla v(s)}^{p-2}\nabla v(s), \nabla u(t) - \nabla v(s)}\\
        & \hspace{.5cm} - \ip{ \abs{\nabla u(t)}^{p-2}\nabla u(t) - \abs{\nabla v(s)}^{p-2}\nabla v(s), \frac{\nabla\delta}{\delta} (u(t) - v(s))} \Bigg)\\
        &\geq \quad -K\frac{\abs{\nabla \delta}^p}{\delta^p}\abs{u(t) - v(s)}^p.
    \end{align*}
    By a similar computation as above we get
    \begin{align*}
        & \quad \, \int \limits_0^\T \!\!\!  \int \limits_0^\T \!\!\! \int \limits_\Om \frac{\zeta_\varepsilon(t,s)}{\delta}\eta''\left( \frac{u(t)-v(s)}{\delta} \right) \frac{\abs{\nabla \delta}^p}{\delta^p}\abs{u(t) - v(s)}^p \,\d x \d t \d s \\
        & \leq \quad 2^p\int \limits_0^\T \!\!\!  \int \limits_0^\T \zeta_\varepsilon(t,s) \d t \d s \int \limits_\Om \frac{\abs{\nabla \delta}^p}{\delta} \,\d x \leq 2^p \T\norm{\zeta}_\infty \int \limits_\Om \frac{\abs{\nabla \delta}^p}{\delta} \,\d x.
    \end{align*}
    A combination of the previous two displays yields the desired inequality for $1 < p < 2$, and the proof is completed.
\end{proof}

By applying these estimates to the equation \eqref{PDE} we get the following inequality consisting of a harmless left-hand side and a delicate error term in the right-hand side.

\begin{proposition}\label{prop:error-3}
    Let $u$ be a weak subsolution and $v$ is a weak supersolution to \eqref{PDE}, respectively. Let $\psi \in L^p(0,\T;W^{1,p}(\Om)) \cap C(\overline{\Om_\T})$ such that $$u \leq \psi \leq v \text{ on }\partial_\mathcal{P} \Om_\T.$$
    Finally, let $\delta : \overline{\Om} \to \R$ be a positive Lipschitz function and $\zeta \in C_0^\infty(0,\T)$ be non-negative. Suppose that, for $\varepsilon > 0$,
    \[
        \zeta_\varepsilon(t,s) \in C_0^\infty((0,\T) \times (0,\T)),
    \]
    and also that
    \[
    x \mapsto \eta'\left( \frac{u(x,t) - v(x,s)}{\delta(x)} \right)\zeta_\varepsilon(t,s)
    \]
    lies in $W_0^{1,p}(\Om)$ for almost every $t,s \in (0,\T)$.
    Then we have
    \[
        \int \limits_0^\T \!\!\!  \int \limits_0^\T \!\!\! \int \limits_\Om \left( \abs{u(t)}^{p-2}u(t) - \abs{v(s)}^{p-2}v(s) \right)^+ \zeta'(t)\phi_\varepsilon(t-s)\,\d x \d t \d s \geq -C(E_{\mathrm T} + E_{\mathrm S}).
    \]
    Here the terms $E_{\mathrm T}$, $E_{\mathrm S}$, and the dependence of the constant $C>0$ on the data, are as in Lemmas \ref{lemma:error-1} and \ref{lemma:error-2}.
\end{proposition}

\begin{proof}
    By Lemma \ref{lemma:max-principle}, $u$ is essentially bounded from above and $v$ is essentially bounded from below, with bounds depending only on $\psi$. Thus Lemma \ref{lemma:error-1} applies, and the claim follows by estimating the inequality in Proposition \ref{prop:doubling} using Lemmas \ref{lemma:error-1} and \ref{lemma:error-2}.
\end{proof}

We finish this section in verifying that the triple integral in the left-hand side of Proposition \ref{prop:error-3} is indeed harmless.

\begin{proposition}\label{prop:harmless}
    Let $u,v \in L^p(\Om_\T)$ and $\zeta \in C_0^\infty(0,\T)$. Then it holds that, as $\varepsilon \downarrow 0$, the integrals
    \[
        \int \limits_0^\T \!\!\!  \int \limits_0^\T \!\!\! \int \limits_\Om \left( \abs{u(t)}^{p-2}u(t) - \abs{v(s)}^{p-2}v(s) \right)^+ \zeta'(t)\phi_\varepsilon(t-s)\,\d x \d t \d s
    \]
    converge to the integral
    \[
        \int \limits_0^\T \!\!\! \int \limits_\Om \left( \abs{u}^{p-2}u - \abs{v}^{p-2}v \right)^+ \zeta' \,\d x \d t.
    \]
\end{proposition}

\begin{proof}
    Take $\varepsilon > 0$ small enough so that the integrand in the triple integral in the statement is supported on $\Om \times (c,\T-c) \times (c,\T-c)$ for $c>0$. By a change of variable $(t,s)=(t,t-z)$ the triple integral can be rewritten as
    \[
      \int\limits_{-\varepsilon}^\varepsilon \frac{1}{\varepsilon}\phi\left( \frac{z}{\varepsilon} \right) \textup{I}(z) \, \d z,
    \]
    where $\textup{I}(z)$ denotes the integral
    \[
       \int \limits_0^\T \!\!\!\int \limits_\Om  
              \left( \abs{u(t)}^{p-2}u(t) - \abs{v(t-z)}^{p-2}v(t-z) \right)^+ \zeta'(t) \d x \d t.
    \]
    It follows from the continuity of the translation in $L^p(\R^{n+1})$ that $\textup{I}(z)$ is continuous at $z = 0$. Hence, because $\phi$ is the standard mollifier, the triple integral converges as $\varepsilon \downarrow 0$ to $\textup{I}(0)$ which is exactly the double integral in the claim.
\end{proof}

\section{Estimates on error terms}\label{sec:estimate-errors}
Our next step towards the comparison principle is to further simplify the error terms $E_{\mathrm T}$ and $E_{\mathrm S}$ from Section \ref{sec:errors} by giving explicit choices for $\delta$. These choices are designed to optimize the bounds on the error terms while keeping the treatment of the lateral boundary values tractable. Below $\delta_U$ denotes upper value and $\delta_L$ lower value.

Because the estimates in Lemma \ref{lemma:error-2} significantly differ between the singular and the degenerate ranges of $p$, we treat these cases separately.

\begin{lemma}\label{lemma:delta-1}
    Assume $1 < p < 2$.
    Let $0 < \delta_L < \delta_U$ and $D > 0$ be numbers, and let $\delta : \overline{\Om} \to \R$ be given by
    \[
       \delta(x)^{(p-1)/p} := \delta_L^{(p-1)/p} + \left(\delta_U^{(p-1)/p} - \delta_L^{(p-1)/p}\right)\left(1 - \frac{\dist\left(x,\partial \Om\right)}{D}  \right)^+
    \]
    Then $\delta$ is a positive Lipschitz function satisfying the estimates
    \[
        \int \limits_{\Om} \delta^{p-1} \,\d x \leq C\left(\delta_L^{p-1} + D \delta_U^{p-1} \right)
    \]
    and
    \[
    \int \limits_{\Om} \frac{\abs{\nabla \delta}^p}{\delta} \,\d x \leq C(\delta_U/D)^{p-1}.
    \]
    The constant $C>0$ depends only on $p$ and $\Om$.
\end{lemma}

\begin{proof}
    Since $x\mapsto\dist(x,\partial \Om)$ is Lipschitz it is easy to see that $\delta : \overline{\Om} \to \R$ is a positive Lipschitz function.
    
    Let us denote $\Om_D := \{ x \in \Om : \dist(x,\partial \Om) < D \}$. Note that $\delta_L \leq \delta \leq \delta_U$ on $\overline{\Om}$ and $\delta = \delta_L$ on $\Om \setminus \Om_D$. Using these we can estimate
    \begin{align*}
        \int \limits_{\Om} \delta^{p-1} \,\d x & = \int \limits_{\Om \setminus \Om_D} \delta_L^{p-1} \,\d x + \int \limits_{\Om_D} \delta^{p-1} \,\d x \leq \abs{\Om}\delta_L^{p-1}  + \abs{\Om_D} \delta_U^{p-1} \\
        & \leq C\left(\delta_L^{p-1 } + D \delta_U^{p-1}\right).
    \end{align*}
    In the last inequality we are using an elementary property $\abs{\Om_D} \leq L D$ of a Lipschitz domain.
    This proves the first inequality. For the second inequality, a chain rule gives
    \[
         \frac{\abs{\nabla \delta}^p}{\delta} =  (p/(p-1))^p \abs{\nabla \! \left(\delta^{(p-1)/p}\right) }^p \leq (p/(p-1))^p \mathds{1}_{\Om_D} \delta_U^{p-1}D^{-p} 
    \]
    which allows us to estimate
    \begin{align*}
        \int \limits_{\Om} \frac{\abs{\nabla \delta}^p}{\delta} \,\d x \leq (p/(p-1))^p \abs{\Om_D} \delta_U^{p-1}D^{-p} \leq C \left( \delta_U/D \right)^{p-1}.
    \end{align*}
    We used $\abs{\Om_D} \leq LD$ in the last inequality. This completes the proof.
\end{proof}

\begin{lemma}\label{lemma:delta-2}
    Assume $p\geq 2$.
        Let $0 < \delta_L < \delta_U$ and $D > 0$ be numbers, and let $\delta : \overline{\Om} \to \R$ be given by
    \[
       \sqrt{\delta(x)} := \sqrt{\delta_L} + \left(\sqrt{\delta_U} - \sqrt{\delta_L}\right)\left(1 - \frac{\dist\left(x,\partial \Om\right)}{D}  \right)^+.
    \]
    Then $\delta$ is a positive Lipschitz function satisfying the estimates
    \[
        \int \limits_{\Om} \delta \,\d x \leq C\left(\delta_L + D \delta_U\right)
    \]
    and
    \[
    \int \limits_{\Om} \frac{\abs{\nabla \delta}^p}{\delta^{p/2}} \,\d x \leq C\delta_U^{p/2}D^{1-p}.
    \]
    The constant $C>0$ depends only on $p$ and $\Om$.
\end{lemma}

\begin{proof}
    The first estimate follows from an identical argument as in the proof of Lemma \ref{lemma:delta-1}. For the second, we first compute
    \[
        \frac{\abs{\nabla \delta}^p}{\delta^{p/2}} = 2^p \abs{\nabla \sqrt{\delta}}^p
    \]
    and then estimate as above.
\end{proof}

And finally, we derive explicit bounds for the error terms.
\begin{proposition}\label{prop:bound-I}
    Let $0 < \delta_L < \delta_U$ and $D > 0$, and let $\delta : \overline{\Om} \to \R$ be as in Lemma \ref{lemma:delta-1} for $1 < p < 2$ and as in Lemma \ref{lemma:delta-2} for $p\geq 2$. Also let $\zeta \in C_0^\infty(0,\T)$ and let $\varepsilon > 0$ be small enough. Then the temporal error term $E_{\mathrm T}$ as in Lemma \ref{lemma:error-1} can be estimated
    \[
        E_{\mathrm T} \leq \frac{L}{\varepsilon} \left( \delta_L^{\min\{p-1,1\}} + D\delta_U^{\min\{p-1,1\}} \right).
    \]
    Furthermore, the spatial error term $E_{\mathrm S}$ as in Lemma \ref{lemma:error-2} has the upper bound
    \[
        E_{\mathrm S} \leq \begin{cases}
            L \left(\dfrac{\delta_U}{D}\right)^{p-1} & \text{ when } 1 < p < 2,\\
            L \left(\dfrac{\delta_U^{p/2}}{D^{p-1}}\right)^{2/p} & \text{ when } p\geq 2.
        \end{cases}
    \]
    The constant $L > 0$ only on $p$, $\T$, $\Om$ and $\zeta$.
\end{proposition}

\begin{proof}
    The estimate for $E_{\mathrm S}$ immediately follows from Lemmas \ref{lemma:delta-1} and \ref{lemma:delta-2}. It remains to estimate $E_{\mathrm T}$. Recall from Lemma~\ref{lemma:error-1} that
    \[
        E_{\mathrm T} = \int \limits_0^\T \!\!\!  \int \limits_0^\T \!\!\! \int \limits_\Om \delta^{\min\{ p-1,1 \}}\left( \abs{\phi_\varepsilon'(t-s)}\abs{\zeta(t)} + \phi_\varepsilon(t-s)\abs{\zeta'(t)} \right) \,\d x \d t \d s.
    \]
    Let $0<\varepsilon<1$. By taking $K \geq 1$ so that $\int \abs{\phi'}\d \tau \leq K$ we can estimate
    \begin{align*}
        E_{\mathrm T} & \leq \frac{K \T}{\varepsilon}\left(\norm{\zeta}_\infty + \norm{\zeta'}_\infty\right) \int \limits_\Om \abs{\delta}^{\min\{p-1,1\}} \,\d x\\
        & \leq \frac{L}{\varepsilon} \left(\delta_L^{\min \{ p-1,1 \} } + D \delta_U^{\min \{ p-1,1 \} }\right).
    \end{align*}
    We used $\int \abs{\phi_\varepsilon'} \d \tau \leq K/\varepsilon$, and $\int \phi_\varepsilon \d \tau = 1$ in the first inequality, and Lemmas \ref{lemma:delta-1} and \ref{lemma:delta-2} in the second inequality. This completes the proof.
\end{proof} 

\section{Comparison principle and uniqueness}\label{sec:comparison}
Our main results are the following uniqueness and comparison results, Theorem~\ref{thm:unique} and Theorem~\ref{thm:comparison-principle}. In both theorems one may take
\[
    \alpha_0=
    \begin{cases}
        1/p, & 1<p<2,\\[2pt]
        (2p-2)/(3p-2), & p\geq2.
    \end{cases}
\]
 Boundary inequalities, the Cauchy--Dirichlet problem and time H\"older continuity of the boundary data are understood in the sense specified in Section~\ref{sec:Preli}.

\begin{theorem}[Uniqueness]
\label{thm:unique}
    There is $\alpha_0 \in (0,1)$ so that the following implication holds. If there is $\alpha \in (\alpha_0,1]$ so that
    \[
    \psi \in C(\overline{\Om_\T}) \cap L^p(0,\T;W^{1,p}(\Om))\quad \text{is $\alpha$-H\"older continuous in time,}
    \]
    then a solution to the Cauchy--Dirichlet problem with boundary values $\psi$ is unique.
\end{theorem}

This in particular holds whenever $\psi$ is Lipschitz continuous in time in the sense that $\psi \in L^p(0,\T;W^{1,p}(\Om)) \cap C(\overline{\Om_\T})$ and there is $L > 0$ with
\[
    \abs{\psi(x,t) - \psi(x,s)} \leq L\abs{t-s} \quad \text{for all} \quad x \in \partial \Om \quad \text{and} \quad t,s \in (0,\T).
\]

We actually prove the following comparison principle, in case one may find suitable $\psi$ in between. Of course, if $u$ or $v$ is of suitable regularity up to the boundary, then naturally one of them can act as $\psi$. The comparison principle then immediately implies the above uniqueness.
\begin{theorem}[Comparison]\label{thm:comparison-principle}
    Let $u$ be a weak subsolution and $v$ be a weak supersolution to \eqref{PDE}, respectively. Then there is $\alpha_0 \in (0,1)$ so that the following implication holds. Suppose that 
    $$
    \psi \in C(\overline{\Om_\T}) \cap L^p(0,\T;W^{1,p}(\Om))\quad \text{is $\alpha$-H\"older continuous in time}
    $$  for some $\alpha \in (\alpha_0,1]$ and satisfies 
    $$
    u \leq \psi \leq v\text{ on }\partial_\mathcal{P}\Om_\T.
    $$ 
    Then $u \leq v$ almost everywhere on $\Om_\T$.
\end{theorem}

In order to prove the comparison principle we choose in a suitable way the parameters $\delta_L,\delta_U$ and $D$ in terms of $\varepsilon$ and an auxiliary parameter $c>0$. We first let $\varepsilon\downarrow0$ with $c$ fixed, and then let $c\downarrow0$, so that both error terms vanish. Before proceeding to the proof, we provide a brief explanation about our choices for the convenience of the reader, when $p \geq 2$.

Let $\varepsilon > 0$ be a small number.
First of all, to even begin with, we need the functions
\begin{equation}\label{eq:test-function}
        x \mapsto \eta'\left( \frac{u(x,t) - v(x,s)}{\delta(x)} \right)\zeta_\varepsilon(t,s)
\end{equation}
to lie in $W_0^{1,p}(\Om)$ for almost every $t,s \in (0,\T)$ so that the inequalities in Proposition \ref{prop:error-3} are valid. To this end, assuming $\psi$ satisfies the condition in Theorem \ref{thm:comparison-principle}, we need $\delta_U \sim \varepsilon^\alpha$.
On the other hand, the upper bound of the temporal error term $E_{\mathrm T}$ in Proposition \ref{prop:bound-I} suggests the choice $\delta_L \ll \varepsilon$. In particular, we necessarily would have $\delta_L \ll \delta_U$. Nevertheless, the upper bound of $E_{\mathrm T}$ can balanced by choosing $D = \delta_L/ \delta_U$.
This yields the required control over the temporal error term $E_{\mathrm T}$. And finally, provided that the exponent $\alpha$ is not too small, these choices of $\delta_U$ and $D$ allow us to control the upper bound of the spatial error term $E_{\mathrm S}$ in Proposition \ref{prop:bound-I} as well.

\begin{proof}[Proof of Theorem \ref{thm:comparison-principle}]
    The value $\alpha_0$ is given in the end of the proof. For now, we fix $\psi$ and $\alpha$ as in the statement, and let $H>0$ be a corresponding H\"older constant on the lateral boundary.
    In order to prove the claim $u \leq v$, it suffices to show that
    \begin{equation}\label{eq:contraction}
        \int \limits_0^\T \!\!\! \int \limits_\Om \left( \abs{u}^{p-2}u - \abs{v}^{p-2}v \right)^+ \zeta' \,\d x \d t \geq 0
    \end{equation}
    for all non-negative $\zeta \in C_0^\infty(0,\T)$. This is due to the following standard argument. Indeed, by proving \eqref{eq:contraction} and applying it to cutoff functions $\zeta$, we get for all $\tau \in (0,\T)$ and small enough $h>0$ that
    \begin{align*}
        \frac{1}{h}& \int \limits_\tau^{\tau+h}\!\!\! \int \limits_\Om \left( \abs{u}^{p-2}u - \abs{v}^{p-2}v \right)^+ \d x\d t\\
        &\leq \quad \frac{1}{h}\int \limits_h^{2h}\!\!\! \int \limits_\Om \left( \abs{u}^{p-2}u - \abs{v}^{p-2}v \right)^+ \d x\d t\\
        &\leq \quad \frac{1}{h}\int \limits_h^{2h}\!\!\! \int \limits_\Om \left( \abs{u}^{p-2}u - \abs{\psi}^{p-2}\psi \right)^+ + \left( \abs{\psi}^{p-2}\psi - \abs{v}^{p-2}v \right)^+ \, \d x \d t.
    \end{align*}
    By applying the Lebesgue differentiation theorem to the first row and the second condition in Definition \ref{def:compare} to the last row, upon sending $h \downarrow 0$ we get 
    \[
        \int \limits_\Om \left( \abs{u(\tau)}^{p-2}u(\tau) - \abs{v(\tau)}^{p-2}v(\tau) \right)^+ \d x \leq 0
    \]
    for almost every $\tau \in (0,\T)$.
    But then we have $(\abs{u}^{p-2}u - \abs{v}^{p-2}v)^+ = 0$ almost everywhere on $\Om_\T$ by Fubini's theorem. Hence $u \leq v$.

    In order to prove \eqref{eq:contraction}, fix a non-negative $\zeta \in C_0^\infty(0,\T)$ and a sufficiently small $c>0$. 
     We choose the remaining parameters according to
    \begin{equation*}
     \delta_U := H \varepsilon^\alpha, \quad \delta_L^{\min\{ p-1,1 \}} := c \cdot \varepsilon \quad \text{and} \quad D \delta_U^{ \min\{ p-1,1 \} } := \delta_L^{\min \{p-1,1\} },
    \end{equation*}
    and let $\delta : \overline{\Om} \to \R$ be as in Lemma \ref{lemma:delta-1} for $1 < p < 2$ and as in Lemma \ref{lemma:delta-2} for $p\geq 2$. Note that by taking $c,\varepsilon$ small enough we have $\delta_L < \delta_U$.

    Now we verify the hypotheses of Proposition \ref{prop:error-3}. Since $\zeta \in C_0^\infty(0,\T)$ it follows from the definition of $\zeta_\varepsilon$, i.e.\ $
\zeta_\varepsilon(t,s)=\zeta(t)\phi_\varepsilon(t-s)$, that $\zeta_\varepsilon \in C_0^\infty( (0,\T) \times (0,\T) )$ for all small enough $\varepsilon$. Next we check that the function in \eqref{eq:test-function} lies in $W_{0}^{1,p}(\Om)$ for almost every $t,s \in (0,\T)$. If $\abs{t-s} \geq \varepsilon$ then this follows from $\zeta_\varepsilon(t,s)=0$. Suppose that $\abs{t-s} < \varepsilon$. For $x \in \Om$ we estimate
    \[
        u(x,t) - v(x,s) \leq (u(x,t) - \psi(x,t))^+ + (\psi(x,s) - v(x,s))^+ + \abs{\psi(x,t) - \psi(x,s)}.
    \]
    By the hypotheses, the first two terms in the right-hand side belong to $W_{0}^{1,p}(\Om)$ for almost every $t,s \in (0,\T)$. The remaining term can be estimated
    \[
        \frac{\abs{\psi(x,t) - \psi(x,s)}}{\delta(x)}
        \leq \frac{H \abs{t-s}^\alpha}{\delta_U} < \frac{H \varepsilon^\alpha}{\delta_U} = 1 \quad \text{for every } x \in \partial \Om.
    \]
    Here we used the fact that $\delta = \delta_U$ on $\partial \Om$. Hence, by recalling that $\eta'(z) = 0$ for $z \leq 1$, we deduce the function in \eqref{eq:test-function} to lie in $W_0^{1,p}(\Om)$ it the remaining cases. Proposition \ref{prop:error-3} is therefore applicable.
    By Proposition \ref{prop:error-3} we now get the inequality
    \begin{align*}
        \int \limits_0^\T &\!\!\! \int \limits_0^\T \!\!\! \int \limits_\Om \left( \abs{u(t)}^{p-2}u(t) - \abs{v(s)}^{p-2}v(s) \right)^+ \zeta'(t)\phi_\varepsilon(t-s)\,\d x \d t \d s\\
        &\geq \quad -L_1(E_{\mathrm T} + E_{\mathrm S}).
    \end{align*}
    
    Suppose first that $1<p<2$.
    By Proposition \ref{prop:bound-I} and the explicit choices of the parameters we have
    \[
        E_{\mathrm T} + E_{\mathrm S} \leq L_2\left( c + 
        \left(\frac{ \varepsilon^{p \cdot \alpha - 1}}{c}\right)^{p-1} \right).
    \]
    Note that here we simply rewrote
    \[
        \frac{\delta_U}{D} = \left(D \delta_U^{p-1}\right)^{-1} \delta_U^p = H^p \frac{\varepsilon^{p \cdot \alpha - 1}}{c}.
    \]
    If $\alpha>1/p$, the spatial error tends to zero as $\varepsilon\downarrow0$ with $c>0$ fixed, while the temporal error is bounded by $Cc$, with $C$ independent of $c$. Proposition \ref{prop:harmless} therefore gives
    \[
        \int \limits_0^\T\!\!\!\int \limits_\Om
        \left(\abs{u}^{p-2}u-\abs{v}^{p-2}v\right)^+
        \zeta'(t)\,\d x\d t
        \geq-Cc.
    \]
    Finally, letting $c\downarrow0$ yields \eqref{eq:contraction}.

    Then consider the case $p \geq 2$. This time we get
    \[
    E_{\mathrm T} + E_{\mathrm S} \leq L_3\left( c + 
        \left(\frac{ \varepsilon^{p\cdot \alpha/2 - (1-\alpha)(p-1)}}{c^{p-1}}\right)^{2/p} \right)
    \]
    by rewriting
    \[
        \frac{\delta_U^{p/2}}{D^{p-1}} = H^{3p/2 - 1} \frac{\varepsilon^{p\cdot \alpha/2 - (1-\alpha)(p-1)}}{c^{p-1}}.
    \]
    Again, the upper bound can be controlled when the exponent of $\varepsilon$ is positive. This holds whenever
    \[
         \alpha > \frac{2p-2}{3p-2}.
    \]
    The inequality \eqref{eq:contraction} for $p \geq 2$ now follows from the same argument as for the range $1 < p < 2$. As noted in the beginning of the proof, it immediately follows that $u \leq v$. The proof of the comparison principle is now complete. As observed, the constant $\alpha_0$ can be chosen
    \[
        \alpha_0 :=
        \begin{cases}
        1/p & \text{ if } 1 < p < 2,\\
        (2p-2)/(3p - 2) & \text{ if } p\geq 2.
        \end{cases}
    \]
\end{proof}

\begin{proof}[Proof of Theorem \ref{thm:unique}]
    Let $u$ and $v$ be solutions to the Cauchy--Dirichlet problem with boundary values $\psi$ where $\psi$ is as in the hypotheses.
    Note that the conditions in Definition \ref{def:cauchy-dirichlet} imply that $u \leq \psi \leq v$ on $\partial_\mathcal{P} \Om_\T$ and $v \leq \psi \leq u$ on $\partial_\mathcal{P} \Om_\T$.
    The uniqueness now follows from the comparison principle in Theorem \ref{thm:comparison-principle}.
\end{proof}

\appendix

\section{Maximum principle}

We recall the \emph{maximum principle}, cf.\ footnote in \cite[Section~2]{LPSSarsa}, of the equation \eqref{PDE} which is used to obtain uniform constants when we apply Lemma \ref{lemma:error-1}.

\begin{lemma}\label{lemma:max-principle}
    Let $u$ be a weak subsolution and $v$ be a weak supersolution to \eqref{PDE}, respectively. Also let $\psi \in C(\overline{\Om_\T}) \cap L^p(0,\T;W^{1,p}(\Om))$ such that $u \leq \psi \leq v$ on $\partial_\mathcal{P}\Om_\T$.
    Then it holds that 
    \[
        \esssup_{\Om_\T} u \leq \max_{\partial_{\mathcal{P}} \Om_\T} \psi \quad \text{and} \quad \min_{\partial_{\mathcal{P}}\Om_\T} \psi \leq \essinf_{\Om_\T} v.
    \]
\end{lemma}

\begin{proof}
   It suffices to prove the part of the claim involving the subsolution $u$ by taking the subsolution $-v$.
    For $\delta>0$ we define
    \[
        H_\delta(z) := \min\{ \max\{0,z/\delta\},1 \}
    \]
    and also
    \[
        G_\delta(z,w) := (p-1)\int \limits_{w}^z H_\delta(\tau-w)\abs{\tau}^{p-2}\,\d \tau.
    \]
    Let $M \in \R$ denote the maximum of $\psi$ on $\partial_\mathcal{P}\Om_\T$. 
    Since $u$ is a weak subsolution to \eqref{PDE} it holds for every non-negative $\zeta \in C_0^\infty(\Om_\T)$ that
    \[
        \int \limits_0^\T \!\!\! \int \limits_\Om \left(\abs{u}^{p-2}u - \abs{M}^{p-2}M\right) \! \partial_t \zeta - \ip{\abs{\nabla u}^{p-2}\nabla u, \nabla \zeta } \d x \d t \geq 0.
    \]
    Now, the idea to reach the conclusion is to test the above equation with functions $\varphi(x,t) := H_\delta(u(x,t)-M)\zeta(t)$ where $\zeta \in C_0^\infty(0,\T)$ is non-negative. Indeed, $\varphi$ vanishes on $\partial\Om \times (0,\T)$ because $u\leq M$ on $\partial_\mathcal{P}\Om_\T$. However, note that $\varphi$ is not regular enough to be a test function because it does not have a time derivative. This issue is handled by regularizing $\varphi$ using Stekloff average, and then applying a convexity inequality\footnote{It is crucial here that $M$ is a constant and not a general function.} as in \cite[Lemma 1]{Otto} and also in \cite[Lemma 3.2 and Lemma 3.3]{ALP26}. This yields
    \[
        \int \limits_0^\T \!\!\! \int \limits_\Om G_\delta(u,M)\zeta' \,\d x \d t \geq \int \limits_0^\T \!\!\! \int \limits_\Om \zeta H_\delta'(u-M) \abs{\nabla u}^p \, \d x \d t \geq 0.
    \]
    By sending $\delta \downarrow 0$ it follows that
    \[
        \int \limits_0^\T \!\!\! \int \limits_\Om \left(\abs{u}^{p-2}u - \abs{M}^{p-2}M\right)^+\! \zeta' \geq 0.
    \]
    Since $\psi$ is continuous on $\overline{\Om_\T}$ and $\psi(x,0)\leq M$, the quantity
    \[
        \left(\abs{\psi(x,t)}^{p-2}\psi(x,t)-\abs{M}^{p-2}M\right)^+
    \]
    tends uniformly to zero as $t\downarrow0$. Together with Definition \ref{def:compare} and the inequality
    \[
        \begin{aligned}
            \left(\abs{u}^{p-2}u-\abs{M}^{p-2}M\right)^+
            &\leq \left(\abs{u}^{p-2}u-\abs{\psi}^{p-2}\psi\right)^+\\
            &\quad+\left(\abs{\psi}^{p-2}\psi-\abs{M}^{p-2}M\right)^+,
        \end{aligned}
    \]
    this shows that the spatial integral of the left-hand side has vanishing initial averages. The inequality obtained after letting $\delta\downarrow0$ gives this integral a nonincreasing representative. Hence it vanishes almost everywhere in $(0,\T)$, proving $u\leq M$ almost everywhere in $\Om_\T$.
\end{proof}

We note that the maximum principle in Lemma~\ref{lemma:max-principle} holds pointwise in the following sense. Since we have the one-sided bounds in Lemma~\ref{lemma:max-principle}, it follows by applying \cite[Theorem 2.3]{Lia21} to the supersolutions $-u$ and $v$ that $u$ and $v$ admit upper- and lower semicontinuous representatives
\[
    \begin{aligned}
        u^*(x,t)&=\operatorname*{ess\,lim\,sup}_{(y,s)\to(x,t)}u(y,s),\\
        v_*(x,t)&=\operatorname*{ess\,lim\,inf}_{(y,s)\to(x,t)}v(y,s).
    \end{aligned}
\]
These satisfy
\[
    \sup_{\Om_\T}u^*=\esssup_{\Om_\T}u,
    \qquad
    \inf_{\Om_\T}v_*=\essinf_{\Om_\T}v.
\]
Consequently, the conclusion of Lemma~\ref{lemma:max-principle} holds pointwise:
\[
    u^*(x,t)\leq\max_{\partial_{\mathcal P}\Om_\T}\psi,
    \qquad
    \min_{\partial_{\mathcal P}\Om_\T}\psi\leq v_*(x,t)
\]
for every $(x,t)\in\Om_\T$.

\section{Algebraic inequalities}
Here we derive some useful auxiliary estimates. We recall from Section~\ref{sec:Otto} that, for $\delta>0$,
\[
    q_\delta(z,w) := (p-1)\int \limits_w^z \eta'\left(\frac{\tau-w}{\delta}\right) \abs{\tau}^{p-2}\, \d \tau.
\]

\begin{lemma}\label{lemma:q-error}
    Let $M > 0$ and $\delta > 0$. Suppose that $z\leq M$ and $w \geq -M$. Then it holds that
    \[
        0 \leq \left(\abs{z}^{p-2}z - \abs{w}^{p-2}w\right)^+ - q_\delta(z,w) \leq K\delta^{ \min\{p-1,1\} }.
    \]
    The constant $K > 0$ depends only on $M$ and $p$. The same estimate holds if we replace $q_\delta(z,w)$ with $\tilde{q}_\delta(w,z)$.
\end{lemma}

\begin{proof}
    Since $\tilde{q}_\delta(w,z) = q_\delta(-w,-z)$ it suffices to prove the estimates for $q_\delta$.
    If $z\leq w$, then $q_\delta(z,w)=(\abs{z}^{p-2}z-\abs{w}^{p-2}w)^+=0$. Hence, we may assume that $z>w$. In particular,  it holds that $z,w \in [-M,M]$. Since $0\leq\eta'\leq1$, we have
    \[
        0\leq q_\delta(z,w)\leq(p-1)\int\limits_w^z\abs{\tau}^{p-2}\,\d\tau
        =\abs{z}^{p-2}z-\abs{w}^{p-2}w.
    \]
    This proves the first inequality in the claim. But then there is $H>0$ depending only on $M$ and $p$ so that
    \begin{equation}\label{eq:b-holder}
        \abs{\abs{z}^{p-2}z - \abs{w}^{p-2}w} \leq H \abs{z-w}^{\min\{ p-1,1 \}}.
    \end{equation}
    The remaining cases are dealt using \eqref{eq:b-holder}.
    
    In the case $w < z < w + 2\delta$ we first estimate
    \begin{align*}
        \left(\abs{z}^{p-2}z - \abs{w}^{p-2}w\right)^+ - q_\delta(z,w) & \leq \abs{\abs{z}^{p-2}z - \abs{w}^{p-2}w},
    \end{align*}
    and then apply \eqref{eq:b-holder} together with $\abs{z-w} \leq 2\delta$.
    Lastly, assume $z \geq w + 2\delta$. By the properties of $\eta$ it follows that
   \begin{align*}
       (p-1)^{-1}&\left\{ (\abs{z}^{p-2}z - \abs{w}^{p-2}w)^+ - q_\delta(z,w) \right\} \\
       &= \enspace\int \limits_{w}^z \abs{\tau}^{p-2} \d \tau -\left\{  \int \limits_{w + \delta}^{w + 2\delta} \eta'\left( \frac{\tau - w}{\delta} \right) \abs{\tau}^{p-2} \,\d \tau +  \int \limits_{w + 2\delta}^{z} \abs{\tau}^{p-2} \,\d \tau \right\}\\
       &\leq \enspace \int \limits_{w}^{w+2\delta} \abs{\tau}^{p-2} \d \tau = (p-1)^{-1}\{ \abs{w + 2\delta}^{p-2}(w + 2\delta) - \abs{w}^{p-2}w \}.
   \end{align*}
   The estimate \eqref{eq:b-holder} then closes the proof.
\end{proof}

We do not need this fact here but the constant $K$ in Lemma \ref{lemma:q-error} can be chosen to not depend on $M$ when $1 < p \leq 2$ because $z \mapsto |z|^{p-2}z$ would then be globally Hölder continuous.

Next we state some algebraic inequalities that provide sharper estimates than the basic monotonicity inequality $\langle |a|^{p-2}a-|b|^{p-2}b,a-b\rangle \geq 0$.
This is a crucial detail of the paper.

\begin{lemma}\label{lemma:L-p}
    Let $a,b,c \in \R^n$. If $p \geq 2$ it holds that
    \[
        \ip{\abs{a}^{p-2}a - \abs{b}^{p-2}b, a - b - c} \geq -K \abs{c}^2 \bigr(\abs{a}^p + \abs{b}^p\bigr)^{(p-2)/p}
    \]
    and if $1 < p < 2$ it holds that
    \[
        \ip{\abs{a}^{p-2}a - \abs{b}^{p-2}b, a - b - c} \geq -K \abs{c}^p.
    \]
    The constant $K>0$ depends only on $p$.
\end{lemma}

\begin{proof}
    These estimates can be deduced from \cite[Section 12]{Lin17} but we sketch the details for the convenience of the reader. The case $p=2$ follows easily from Young's inequality. Then assume $p > 2$.  By applying 
    \[
        \left| \abs{a}^{p-2}a - \abs{b}^{p-2}b \right|^2 \leq C_1 \! \ip{ \abs{a}^{p-2}a - \abs{b}^{p-2}b,a - b } \int_0^1 \abs{a + \tau(b-a)}^{p-2} \d\tau
    \]
    we get
    \begin{align*}
        & \quad \, \ip{\abs{a}^{p-2}a - \abs{b}^{p-2}b,c} \leq \abs{\abs{a}^{p-2}a - \abs{b}^{p-2}b} \abs{c}\\
        &\leq \quad C_1 \abs{c} \sqrt{\ip{ \abs{a}^{p-2}a - \abs{b}^{p-2}b,a - b }}\sqrt{\int_0^1 \abs{a + \tau(b-a)}^{p-2} \, \d\tau}\\
        &\leq \quad (1/2)\ip{ \abs{a}^{p-2}a - \abs{b}^{p-2}b,a - b } + C_2\abs{c}^2 \int_0^1 \abs{a + \tau(b-a)}^{p-2} \, \d\tau\\
        &\leq \quad (1/2)\ip{ \abs{a}^{p-2}a - \abs{b}^{p-2}b,a - b } + K\abs{c}^2\left( \abs{a}^p + \abs{b}^p  \right)^{(p-2)/p}.
    \end{align*}
    In the third row we used Young's inequality with exponent $2$, and in the fourth row Hölder's inequality with exponent $p/(p-2)$ and the convexity of $p$-integral. From this we get
    \begin{align*}
        & \quad \, \ip{\abs{a}^{p-2}a - \abs{b}^{p-2}b, a - b - c}\\
        &\geq \quad (1/2) \ip{ \abs{a}^{p-2}a - \abs{b}^{p-2}b,a - b } -  K\abs{c}^2\left( \abs{a}^p + \abs{b}^p  \right)^{(p-2)/p}\\
        &\geq \quad -K\abs{c}^2\left( \abs{a}^p + \abs{b}^p  \right)^{(p-2)/p}.
    \end{align*}
    In the last row we used $\langle\abs{a}^{p-2}a - \abs{b}^{p-2}b,a-b\rangle \geq 0$. This concludes the case $p > 2$.

    It remains to treat $1 < p < 2$. In this range it holds that
    \[
        \abs{\abs{a}^{p-2}a - \abs{b}^{p-2}b}^{p/(p-1)} \leq C_3 \ip{\abs{a}^{p-2}a - \abs{b}^{p-2}b,a-b}.
    \]
    Then we estimate
    \begin{align*}
        & \quad\, \ip{\abs{a}^{p-2}a - \abs{b}^{p-2}b,a-b}\\
        &= \quad \ip{\abs{a}^{p-2}a - \abs{b}^{p-2}b,a-b - c} + \ip{\abs{a}^{p-2}a - \abs{b}^{p-2}b, c}\\
        &\leq \quad \ip{\abs{a}^{p-2}a - \abs{b}^{p-2}b,a-b - c}\\ &
        \hspace{3cm} + (1/(2C_3))\abs{\abs{a}^{p-2}a - \abs{b}^{p-2}b}^{p/(p-1)} + K\abs{c}^p\\
        &\leq \quad \ip{\abs{a}^{p-2}a - \abs{b}^{p-2}b,a-b - c}\\
        & \hspace{3cm} + (1/2)\ip{\abs{a}^{p-2}a - \abs{b}^{p-2}b,a-b} + K \abs{c}^p.
    \end{align*}
    In the third row we used Young's inequality with exponent $p$.
    By absorbing the second term on the last row to the first row, and then applying the inequality $\langle\abs{a}^{p-2}a - \abs{b}^{p-2}b,a-b\rangle \geq 0$, the desired inequality follows. This verifies the remaining case $1 < p < 2$ and the proof is completed.
\end{proof}

\end{document}